\documentclass[11pt]{amsart}
\usepackage[colorlinks=true,linkcolor=blue,hypertexnames=false]{hyperref}

\usepackage{enumerate}
\usepackage{lineno}
\usepackage{graphicx}
\usepackage{geometry}
\usepackage{enumitem}
\makeatletter
\@namedef{subjclassname@}{%
\textup{}}
\makeatother
\newtheorem{thm}{Theorem}[section]

\newtheorem{lem}[thm]{Lemma}

\newtheorem*{conj*}{Denjoy's Conjecture}

\theoremstyle{definition}

\newtheorem{exa}[thm]{Example}
\numberwithin{equation}{section}

\begin{document}

\title[]{On questions posed by Gundersen and Yang concerning a certain binomial differential equation }

\author{  Xuxu Xiang, Jianren Long*}

\address{Xuxu Xiang \newline School of Mathematical Sciences, Guizhou Normal University, Guiyang, 550025, P.R. China. }
\email{1245410002@qq.com}

\address{Jianren Long \newline
School of Mathematical Science, Guizhou  Normal University, Guiyang, 550025,  China.}
\email{longjianren2004@163.com }


\date{}


\begin{abstract}
Suppose \(a, b, c\) are polynomials with \(ab \neq 0\) and \(c\) nonconstant. We determine all entire solutions to the equation
\begin{equation*}
f'f'' - a(z)f^2 = b(z)e^{2c(z)}.
\end{equation*}
This completely solves two open questions posed by Gundersen and Yang in [Comput. Methods Funct. Theory 21 (2021), 605-617].  
\end{abstract}



\keywords{ Nevanlinna theory; Entire solutions; Binomial differential equations; Nonlinear differential equations
*Corresponding author. \\
}
\maketitle
\section{Introduction and main results}\label{sec1}

Let $f$ be a meromorphic function in the complex plane $\mathbb{C}$. We assume    that readers are familiar   with the standard notation and basic results of Nevanlinna theory, 
such as the proximity function $m(r,f)$, the counting function $N(r,f)$, and the 
characteristic function $T(r,f)=m(r,f)+N(r,f)$, see \cite{c}  for more details. 
We denote by $\rho(f)$ 
the order of $f$ and by $\lambda(f)$ the exponent of convergence of the zeros of $f$.  $S(r,f)$ denotes any quantity such that if $\rho(f)<\infty$, then $S(r,f)=O(\log r)$; if $\rho(f)=\infty$, then $S(r,f)=O(\log(rT(r,f)))$ for all $r$ outside an exceptional set of finite logarithmic measure.

 In an important pape \cite{hayman}, Hayman conjectured: If $f$ is meromorphic in the plane and both $f$ and $f^{(k)}$ have no zeros for some $k \geq 2$, then either $f(z) = e^{az + b}$ or $f(z) = (az + b)^{-n}$, where $a$ and $b$ are constants and $n$ is a positive integer.
 This conjecture was proved for $k \geq 3$ by Frank \cite{frank} and for $k=2$ by Langley \cite{langley93}.
 Hayman observed that if $f$ has no zeros, then $f$ can be represented as $f = 1/g$, where $g$ is an entire function. Thus,
\[
f'' = \frac{2(g')^2 - gg''}{g^3},
\]
which leads to the question of under what conditions the expression $2(g')^2 - gg''$ has no zeros. 
Inspired by this, Mues \cite{mues78} determined the explicit forms of entire function $g$ for which $2(g')^2 - agg''$ has no zeros, where $a \neq 1$ is a constant. For meromorphic functions $g$, Bergweiler\cite{b} and others obtained results under certain conditions. These results lead naturally to questions regarding the meromorphic solutions of the differential equation
\begin{align}
\label{1}
ff'' - a(z)(f')^2 = d(z)e^{c(z)},
\end{align}
where $a$ and $d$ are polynomials and $c$ is a nonconstant entire function. When $c$ is a polynomial, the explicit forms of entire solutions $f$ were given by Gundersen and Yang \cite{gy}.

The three differential polynomials
\begin{equation}
\label{2}
ff'' - a(z)(f')^2, \quad f'f'' - a(z)f^2, \quad ff' - a(z)(f'')^2,  
\end{equation}
where \(a(z)\) is a polynomial, have the common characteristics that they are  binomial differential polynomials of degree two, where each of \(f, f', f''\) appears in exactly one term. Regarding the differential polynomials in \eqref{2},  by comparing Eqs. \eqref{1}, Gundersen and Yang \cite{gy} posed five questions; blow we list only the two relevant to this paper.

\textbf{Question 1} \emph{ Suppose \(a, b, c \)  are polynomial \(ab \neq 0\), \(c\) nonconstant. What can be said about the entire solutions of equation \begin{equation}
\label{1.1}
f'f'' - a(z)f^2 = b(z)e^{2c(z)}?
\end{equation} }

\textbf{Question 2} \emph{If \(f\) is an entire solution of an equation of the form \eqref{1.1}, where either \(a(z)\) or \(b(z)\) is non-constant, then does \(f\) necessarily have the form \(f(z) = p(z)e^{c(z)}\) where \(p(z)\) is a polynomial?}

Wu et al.\cite{wll} gave a partial solution to Questions 1 and 2.
 
\begin{thm}[\cite{wll}]
\label{thw}
Let \(f\) be an entire solution of the differential Eq. \eqref{1.1}. If all zeros of \(f'\) are simple, then we have one of the following situations:

(1) There exist at least one of \(a, b, c'\) is not constant, we have three possibilities:

\rm(i) \(b\) is non-constant and \(f = Re^{c}\) where \(R\) is a polynomial satisfying
\[
(R^{\prime} + Rc^{\prime})\left[R^{\prime \prime} + 2R^{\prime}c^{\prime} + Rc^{\prime \prime} + R(c^{\prime})^{2}\right] - aR^{2} = b.
\]

\rm(ii) \(b \equiv \mu\) for some constant \(\mu \neq 0\), one of the following three cases occur:
\begin{itemize}
    \item \(a(z) = c'(z)c''(z) + (c'(z))^3 - A_1\), where \(A_1\) is a non-zero constant, and
    \[
    f(z) = \alpha e^{c(z)}, 
    \]
    where \(\alpha\) is a constant satisfying \(\alpha^{2}A_{1} = \mu\).
    
    \item \(a(z) = c'(z)c''(z) + (c'(z))^{3} - L_{1}(z)\), where \(L_{1}(z)\) is a polynomial of degree 1, and
    \[
    f(z) = R(z)e^{\frac{\lambda}{2} z^2 +\gamma z}, 
    \]
    where \(\lambda (\neq 0)\), \(\gamma\) are constants with \(c'(z) = \lambda z + \gamma\) and \(R(z)\) is a polynomial satisfying
    \[
    R^{\prime}R^{\prime \prime} + 2c^{\prime}(R^{\prime})^{2} + \lambda RR^{\prime} + c^{\prime}RR^{\prime \prime} + 3(c^{\prime})^{2}RR^{\prime} + L_{1}R^{2} = \mu .
    \]
    
    \item \(a(z) = c'(z)c''(z) + (c'(z))^{3} - L_{2}(z)\), where \(L_{2}(z)\) is a polynomial with \(\deg L_{2}(z) \geq 2\), and
    \[
    f(z) = R(z)e^{c(z)},
    \]
    where \(\deg c(z)\geq 3\) and \(R(z)\) is a polynomial satisfying
    \[
    R^{\prime}R^{\prime \prime} + 2(R^{\prime})^{2}c^{\prime} + RR^{\prime}c^{\prime \prime} + RR^{\prime \prime}c^{\prime} + 3RR^{\prime}(c^{\prime})^{2} + L_{2}R^{2} = \mu .
    \]
\end{itemize}

(iii) \(\lambda (f) = \rho (f)< \infty .\)

(2) \(a(z)\equiv \tau\) for some constant \(\tau \neq 0,b(z)\equiv \mu\) for some constant \(\mu \neq 0,c(z) = \lambda z\) for some constant \(\lambda \neq 0\) and one of the following three cases occur:
\[
f(z) = \alpha e^{\lambda z},  
\]
where \(\alpha\) is a constant satisfying \((\lambda^{3} - \tau)\alpha^{2} = \mu .\)
\[
f(z) = c_{1}e^{\lambda z} + c_{2}e^{-\frac{1}{2}\lambda z},  
\]
where \(c_{1},c_{2}\) are non-zero constants such that \(\frac{9}{8} c_{1}^{2}\lambda^{3} = \mu\)
\[
f(z) = c_{1}e^{t_{1}z} + c_{2}e^{t_{2}z},  
\]
where \(t_{1},t_{2}\) are two distinct constants such that \(t_{1}^{3} = t_{2}^{3}\) and \(c_{1},c_{2}\) are non-zero constants such that \(c_{1}c_{2}t_{2}(t_{2} - t_{1})(t_{2} + 2t_{1}) = \mu .\)
\end{thm}

Later, Yang et al. \cite{yss} gave a partial solution to Questions 1 and 2 under the condition that $f$ satisfying $\lambda(f)<\rho(f)$.  Recently Long et al. \cite{Long} removed condition that all zeros of \(f'\) are simple in Theorem \ref{thw}, and obtained the same conclusion as Theorem \ref{thw}. To the best of our knowledge, Questions 1 and 2 have not yet been fully solved. By the following theorem, we completely solve Questions 1 and 2.

 \begin{thm}
 \label{th1}
 An entire function \(f\) satisfies \eqref{1.1}
if and only if it belongs to one of the following three classes:

(I) There exists a nonzero polynomial \(R\), such that
\begin{equation*}
(R' + Rc')(R'' + 2R'c' + Rc'' + R(c')^2) - aR^2 = b,
\end{equation*}
and
\begin{equation*}
f = Re^c.
\end{equation*}

(II) There exist  nonzero constants \(\lambda, \mu, \) and \(\alpha, \beta\), such that
\begin{equation*}
c(z) = \lambda z, \quad b \equiv \mu, \quad a \equiv -\frac{\lambda^3}{8}, \quad \mu = \frac{9}{8}\alpha^2\lambda^3,
\end{equation*}
and
\begin{equation*}
f(z) =\alpha e^{\lambda z} + \beta e^{-\lambda z/2}.
\end{equation*}

(III) There exist  nonzero constants \(\lambda, \mu,  \) and \(\alpha, \beta\), such that
\begin{equation*}
c(z) = \lambda z , \quad b \equiv \mu, \quad a \equiv -8\lambda^3, \quad \mu = 24\alpha\beta\lambda^3,
\end{equation*}
and
\begin{equation*}
f(z) =  \alpha e^{(1+i\sqrt{3})\lambda z} + \beta e^{(1-i\sqrt{3})\lambda z} .
\end{equation*}
 \end{thm}

We give the following example to show that there exists a solution 
 $f$ such that $f'$
  has multiple zeros.
  \begin{exa}
  Let 
$a(z) = 1,  c(z) = z,   R(z) = z^2 - 2z + 2,$
and let
$b(z) = 6z^3 - 8z^2 + 8z - 4,  f(z) = R(z)e^z$.
Then
$f'(z) = z^2e^z,   f''(z) = z(z+2)e^z,$
so \(f'\) has a double zero at \(z=0\);  and
\begin{align*}
f'f'' - f^2 &= [z^3(z+2) - (z^2 - 2z + 2)^2]e^{2z} \\
&= (6z^3 - 8z^2 + 8z - 4)e^{2z} = be^{2c}.
\end{align*}
  \end{exa}

 \section{Lemmas}

\begin{lem}
\label{l1}
Suppose \(S,X,G\) are rational functions with \(XG\neq 0\), and \(a,c\) are nonzero polynomials with \(c\) nonconstant. If
\begin{equation}
\label{3.1}
S X' + (2S' + 3S^2)X + X^2 = 0,
\end{equation}
\begin{equation}
\label{3.2}
8a = 2SS' + X' + S^3 + 3SX,
\end{equation}
\begin{equation}
\label{3.3}
S = 2c' + \frac{G'}{G},
\end{equation}
then \(S\) and \(G\) must be nonzero constants such that
\[
X \equiv -3S^2,\qquad c' \equiv \frac{S}{2}, \qquad a=-S^3.
\]
\end{lem}

\begin{proof}
Dividing \eqref{3.1} by \(X^2\), we obtain
\begin{equation}
\label{3.4}
S Y' - (2S' + 3S^2)Y = 1,
\end{equation}
where \(Y=1/X\).

First we show that every finite pole of \(S\) is simple. Let \(z_0\) be a finite pole of \(S\) of order \(m\ge 1\), and write \(t=z-z_0\). Suppose
\[
S=k t^{-m}+O(t^{-m+1}),\qquad X=x t^{\nu}(1+O(t)),
\]
with \(k\neq 0\), \(x\neq 0\). If \(m>1\), then in \eqref{3.1} the term \(3S^2X\) has order \(\nu-2m\), while \(S X'\) and \(2S'X\) have order at least \(\nu-m-1>\nu-2m\). If \(\nu\ge -2m\), then \(3S^2X\) is the unique lowest-order term, with nonzero coefficient \(3k^2x\), impossible. If \(\nu<-2m\), then \(X^2\) has order \(2\nu<\nu-2m\), again a unique lowest-order term with nonzero coefficient \(x^2\), impossible. Hence \(m=1\).

Thus, at a finite pole \(z_0\) of \(S\), we may write
\begin{equation}
\label{3.5}
S = k t^{-1} + O(1),\quad
X = x t^{\nu} (1 + O(t)),\quad
k \in \mathbb{Z} \setminus \{0\},\quad
x \neq 0,\quad
\nu \in \mathbb{Z}.
\end{equation}
Substituting \eqref{3.5} into \eqref{3.1}, we get
\[
x(k\nu - 2k + 3k^2) t^{\nu-2} + x^2 t^{2\nu} + O(t^{\nu-1}) = 0.
\]
If \(\nu<-2\), the term \(X^2\) gives an uncancelled lowest-order term. If \(\nu>-2\), then \eqref{3.1} forces \(\nu=2-3k\); but then the coefficient of \(t^{-3}\) in \eqref{3.2} is \(k^2(k-2)\), which cannot vanish while \(\nu>-2\). Hence \(\nu=-2\). Comparing the coefficient of \(t^{-4}\) in \eqref{3.1} gives
\[
x = -k(3k-4).
\]
Substituting this into \eqref{3.2}, the coefficient of \(t^{-3}\) on the right-hand side is
\[
k^3 - 2k^2 + (3k - 2)x = -8k(k - 1)^2.
\]
The left-hand side \(8a\) is analytic at finite points, so this coefficient must vanish. Hence \(k=1\), and consequently \(x=1\).
In summary, at every finite pole \(z_0\) of \(S\), we have
\begin{equation}
\label{3.6}
\operatorname{ord}_{z_0} S = -1,\quad
\operatorname{res}_{z_0} S = 1,\quad
\operatorname{ord}_{z_0} X = -2,\quad
\operatorname{ord}_{z_0} Y = 2,
\end{equation}
where \(\text{ord}_{z_0} S\) denotes the order of the pole of \(S\) at \(z_0\), and \(\text{res}_{z_0} S\) denotes the residue of \(S\) at \(z_0\).

Next we control the poles of \(Y\). If \(X(z_0)=0\) and \(S(z_0)\neq 0\), then the lowest-order term in \eqref{3.1} is \(S X'\), which cannot vanish. Hence \(X\) has no zero at any point where \(S\neq 0\). Therefore the finite poles of \(Y=1/X\) can occur only at finite zeros of \(S\). Let \(z_0\) be a zero of \(S\) of multiplicity \(m\), i.e.
\[
S(t)=A t^m(1+O(t)),\qquad A\neq 0,\ m\ge 1,
\]
where \(t=z-z_0\). Suppose that \(Y=1/X\) has a pole of order \(\ell\) at \(z_0\):
\[
Y(t)=B t^{-\ell}(1+O(t)),\qquad B\neq 0,\ \ell\ge 0.
\]
Substituting these expansions into \eqref{3.4} gives
\[
S Y'\sim AB(-\ell)t^{m-\ell-1},
\]
and
\[
(2S'+3S^2)Y\sim 2ABm\,t^{m-\ell-1}+O(t^{m-\ell}).
\]
Thus the lowest-order term on the left-hand side is
\[
AB(-\ell-2m)t^{m-\ell-1}.
\]
Since \(\ell+2m>0\), this term does not vanish. The right-hand side is the constant \(1\), which has no negative powers near \(z_0\). Hence the left-hand side cannot have a negative power either, and therefore
\begin{equation}
\label{3.7}
\ell \le m - 1.
\end{equation}

Since the logarithmic derivative of a rational function is \(O(1/z)\) at infinity, \eqref{3.3} shows that the polynomial part of \(S\) is exactly \(2c'\). Let \(S\) have \(N\) finite poles, and set
\[
n = \deg(2c') = \deg c - 1 \ge 0.
\]
Since every finite pole of \(S\) is simple, and the degree of its polynomial part is \(n\), after writing \(S\) as a reduced fraction, the degree of the denominator is \(N\), and the degree of the numerator is \(N+n\); hence the total multiplicity of the finite zeros of \(S\) is \(N+n\).

If \(N+n\ge 1\), suppose among these finite zeros there are \(R_S\ge 1\) distinct points. By \eqref{3.7},
\begin{equation}
\label{3.8}
\text{the total order of finite poles of }Y \le N+n-R_S.
\end{equation}
On the other hand, \eqref{3.6} gives that the total order of finite zeros of \(Y\) is at least \(2N\). If \(n\ge 1\), and \(S(z)\sim s_n z^n\) at infinity, comparing the highest-order terms at infinity in \eqref{3.4} yields
\begin{equation*}
Y(z)\sim -\frac{1}{3s_n^2} z^{-2n}.
\end{equation*}
Hence \(Y\) has a zero of order \(2n\) at infinity. If \(n=0\), the same comparison shows that the point at infinity is neither a zero nor a pole of \(Y\). Therefore the total order of zeros of \(Y\) is at least \(2N+2n\). Since \(Y\) is rational, its total order of poles equals its total order of zeros, and all its poles are finite. Thus
\[
\text{the total order of finite poles of }Y \ge 2N+2n.
\]
Combining this with \eqref{3.8}, we get
\[
2N+2n \le N+n-R_S,
\]
which contradicts \(N+n\ge 1\) and \(R_S\ge 1\). Therefore
\[
N=n=0.
\]
Thus \(S\) is a nonzero constant. Then \eqref{3.4} becomes
\[
S Y' - 3S^2Y = 1.
\]
Since \(Y\) is rational, its general solution
\[
Y=Ce^{3Sz}-\frac{1}{3S^2}
\]
forces \(C=0\). Hence
\[
Y=-\frac{1}{3S^2},\qquad X=-3S^2.
\]
Finally, \eqref{3.3} implies that \(G'/G\) is constant, so \(G\) is a constant. Then \(S=2c'\), and hence \(c'=S/2\). This proves the lemma.
\end{proof}

Next, we give the lemma of the sufficiency of Theorem \ref{th1}.

\begin{lem}
\label{l2}
    If \(f\) belongs to one of the three classes in Theorem~\ref{th1}, then \(f\) is an entire solution of \eqref{1.1}.
\end{lem}

\begin{proof}
  For class  (I),
\[
f' = (R' + Rc')e^c, \quad f'' = (R'' + 2R'c' + Rc'' + R(c')^2)e^c
\]
and $(R' + Rc')(R'' + 2R'c' + Rc'' + R(c')^2) - aR^2 = b,$ we immediately obtain \eqref{1.1}.

For the two exponential terms,    
\[
f = \alpha e^{t_1 z} + \beta e^{t_2 z},
\]
then
\begin{equation}
\begin{aligned}
\label{3.10}
f'f'' - af^2 ={}& \alpha^2(t_1^3 - a)e^{2t_1 z} + \beta ^2(t_2^3 - a)e^{2t_2 z} \\
&+ \alpha \beta  \left[ t_1 t_2 (t_1 + t_2) - 2a \right] e^{(t_1 + t_2)z}.
\end{aligned}  
\end{equation}

For class (II), take
\[
t_1 = \lambda, \quad t_2 = -\frac{\lambda}{2}, \quad a = -\frac{\lambda^3}{8}.
\]
Then
\[
t_2^3 - a = 0, \quad t_1 t_2 (t_1 + t_2) - 2a = 0, \quad t_1^3 - a = \frac{9}{8}\lambda^3.
\]
Substituting this into \eqref{3.10} , we obtain \(f'f'' - af^2 = \mu e^{2c}\).

In Type (III), take
\[
t_{1,2} = (1 \pm i\sqrt{3})\lambda, \quad a = -8\lambda^3.
\]
Then
\[
t_1^3 = t_2^3 = a, \quad t_1 + t_2 = 2\lambda, \quad t_1 t_2 (t_1 + t_2) - 2a = 24\lambda^3.
\]
From \eqref{3.10} and \(\mu = 24\alpha\beta\lambda^3\), we similarly obtain \eqref{1.1}. 
\end{proof}

\section{Proof of Theorem \ref{th1}}
The proof of the sufficiency of the Theorem is given by Lemma \ref{l2}. Next, we prove the necessity.
By \eqref{1.1} and the logarithmic derivative lemma \cite[Theorem 2.2]{c},
\begin{equation*}
\begin{aligned}
T(r, e^{2c}) &= m(r, e^{2c}) \le m\left(r, \frac{f'f'' - af^2}{bf^2}\right) + 2T(r, f) + O(\log r) \\
&\le 2T(r, f) + S(r, f).
\end{aligned}  
\end{equation*}
Hence \(f\) is a transcendental entire function. Taking the logarithmic derivative of both sides of \eqref{1.1} gives 
\begin{equation}
\label{2.1}
b(f'')^2 + bf'f''' - 2abf f' - df'f'' = Pf^2,
\end{equation}
where  
$d = b' + 2bc',   P = ba' - ab' - 2abc'. $ 
We have \(P \neq 0\). Otherwise
$\frac{a'}{a} - \frac{b'}{b} - 2c' \equiv 0$ implies $ e^{2c} = c_1\frac{a}{b}$, where $c_1$ is a constant,
which contradicts the fact that $a,b$ are polynomial.
By \eqref{2.1}, we get
\begin{equation}
\label{2.2}
\frac{b(f'')^2 + bf'f''' - df'f''}{Pff'} - \frac{2ab}{P} = \frac{f}{f'}.  
\end{equation}
By the logarithmic derivative lemma and \eqref{2.2}, we have
$m\left(r, \frac{f}{f'}\right) = S(r,f).$  
 
Differentiating \eqref{2.1}, and then eliminating \(f^2\) by using \eqref{2.1}, we obtain
\begin{equation}
\label{2.3}
\begin{aligned}
0 ={}& (bP' + 2bc'P)(f'')^2 - 3bP f'' f''' + (bP' + 2bc'P) f' f''' - bP f' f^{(4)} \\
&+ (4a'bP - 4abc'P - 2abP') f f' + 2abP (f')^2 + 2abP f f'' + (d'P - dP') f' f''.
\end{aligned} 
\end{equation}

We rewrite \eqref{2.3} as
\(f'' A = f' B\), where
\begin{equation}
\label{2.4}
A = 2abP f + (d'P - dP')f' + (bP' + 2bc'P)f'' - 3bP f''',  
\end{equation}
\begin{equation}
\label{2.5}
B = (4abc'P + 2abP' - 4a'bP)f - 2abP f' - (bP' + 2bc'P)f''' + bP f^{(4)}.  
\end{equation}
Define  
 \begin{align}
    \label{2.6}
    h = \frac{A}{f'} = \frac{B}{f''}.
 \end{align}
  If \(z_0\) is a simple zero of \(f'\), then \(f''(z_0) \neq 0\); from \(f''A = f'B\) we know \(A(z_0) = 0\), hence \(h=A/f'\) is removable at \(z_0\). If \(z_0\) is a multiple zero of \(f'\), then \(f'(z_0) = f''(z_0) = 0\). Substituting into \eqref{2.1} yields $P(z_0)f^2(z_0)=0$.   If \(P(z_0) \neq 0\), then \(f(z_0) = 0\). Substituting into \eqref{1.1}   yields \(b(z_0) = 0\). Therefore, $N(r,h)=N(r,\frac{A}{f'})=O(\log r).$  By $m\left(r, \frac{f}{f'}\right) = O(\log rT(r,f))$ and    the logarithmic derivative lemma, $m(r,h)=S(r,f)$. Thus $T(r,h)=S(r,f)$. 

 From \eqref{2.6}, we obtain two linear differential equations
\begin{equation}
\label{2.7}
-3bP f''' + (bP' + 2bc'P)f'' + (d'P - dP' - h)f' + 2abP f = 0,  
\end{equation}
\begin{equation}
\label{2.8}
bP f^{(4)} - (bP' + 2bc'P)f''' - h f'' - 2abP f' + (4abc'P + 2abP' - 4a'bP)f = 0.  
\end{equation}

Differentiating \eqref{2.7}, and then eliminating \(f^{(4)}\) by using  \eqref{2.8}, we obtain
\begin{equation}
\label{2.9}
C_3 f''' + C_2 f'' + C_1 f' + C_0 f = 0, 
\end{equation}
where
\begin{equation}
\begin{aligned}
\label{2.10}
C_3 &= -4bc'P - 3b'P - 5bP', \\
C_2 &= bP'' + 4b'c'P + 4bc''P + b''P - 4h, \\
C_1 &= -4abP + d''P - dP'' - h', \\
C_0 &= 12abc'P - 10a'bP + 2ab'P + 8abP'.
\end{aligned}  
\end{equation}

By taking \(C_3 \times \eqref{2.7} + 3bP \times \eqref{2.9}\) to eliminate \(f'''\), we obtain
\begin{equation}
\label{2.11}
Q_2 f'' + Q_1 f' + Q_0 f = 0, 
\end{equation}
where
\begin{equation}
\begin{aligned}
\label{2.12}
Q_2 &= (bP' + 2bc'P)C_3 + 3bPC_2, \\
Q_1 &= (d'P - dP' - h)C_3 + 3bPC_1, \\
Q_0 &= 2abPC_3 + 3bPC_0 \\
&= 28ab^2c'P^2 + 14ab^2PP' - 30a'b^2P^2.
\end{aligned}  
\end{equation}
Here $Q_0 \neq 0.$
In fact, if \(Q_0 \equiv 0\), then $14c' = 15\frac{a'}{a} - 7\frac{P'}{P} $ implies $ e^{14c} = c_2\frac{a^{15}}{P^7}$, where $c_2$ is a constant, which   contradicts the fact that \(c\) is non-constant.

Differentiating \eqref{2.11} and eliminating \(f'''\) by using \eqref{2.7}, we get
\begin{equation}
\label{2.13}
B_2 f'' + B_1 f' + B_0 f = 0, 
\end{equation}
where
\begin{equation}
\begin{aligned}
\label{2.14}
B_2 &= 3bP(Q'_2 + Q_1) + (bP' + 2bc'P)Q_2, \\
B_1 &= 3bP(Q'_1 + Q_0) + (d'P - dP' - h)Q_2, \\
B_0 &= 3bP Q'_0 + 2abP Q_2.
\end{aligned} 
\end{equation}
By \eqref{2.11} and \eqref{2.13}, we can eliminate \(f''\) ,  then we obtain
\begin{equation}
\label{2.15}
u f' = v f,
\end{equation}
where 
\begin{equation}
\begin{aligned}
\label{2.16}
u &=Q_1 B_2 - Q_2 B_1\\
&= 3bP Q_1(Q'_2 + Q_1) + (bP' + 2bc'P)Q_1 Q_2 - 3bP Q_2(Q'_1 + Q_0) - (d'P - dP' - h)Q_2^2, \\
v &=Q_2 B_0 - Q_0 B_2\\
&= Q_2(3bP Q'_0 + 2abP Q_2) - 3bP Q_0(Q'_2 + Q_1) - (bP' + 2bc'P)Q_0 Q_2.  
\end{aligned}
\end{equation}
From \eqref{2.16}, we get $T(r,u)+T(r,v)=S(r,f)$.  

If $uv\not\equiv0$, then $\frac{f'}{f}=\frac{v}{u}$ implies $\overline{N}(r,\frac{1}{f})=S(r,f)$. While the multiple zeros of \(f\) must zero of $b$   by \eqref{1.1}.  Thus 
${N}(r,\frac{1}{f})=S(r,f)$. The first fundamental theorem, \eqref{1.1} and the logarithmic derivative lemma give
\begin{equation*}
\begin{aligned}
2T(r, f) &= T(r, 1/f^2) + O(1) \\
&\le m(r, e^{-2c}) + m\left(r, \frac{f'}{f}\right) + m\left(r, \frac{f''}{f}\right) + S(r, f) \\
&\le O(r^m) + S(r, f).
\end{aligned}  
\end{equation*}
We get $\rho(f) = m < \infty.$
Therefore,  ${N}(r,\frac{1}{f})=O(\log rT(r,f))=O(\log r)$. 
Then the Hadamard factorization gives
$f = R_0 e^W,$ where  $  R_0, W $  are nonzero polynomials. 
Substituting this into \eqref{1.1}, we obtain
\begin{equation}
\label{2.17}
\Phi(z) e^{2W} = b e^{2c},  
\end{equation}
where
\[
\Phi = (R'_0 + R_0 W')(R''_0 + 2R'_0 W' + R_0 W'' + R_0 (W')^2) - a R_0^2
\]
is a nonzero polynomial. Therefore \(e^{2(W-c)} = b/\Phi\) is a rational function.  Hence \(W - c\) is a constant. Absorbing this constant into \(R_0\), we obtain the conclusion (I).

We continue to consider $uv\equiv 0$. First, \(Q_2 \not\equiv 0\). If \(Q_2 \equiv 0\), from
\[
u = Q_1 B_2 - Q_2 B_1
\]
and
\[
B_2 = 3bP(Q'_2 + Q_1) + (bP' + 2bc'P)Q_2,
\]
we obtain \(u = 3bP Q_1^2=0\), so \(Q_1 \equiv 0\); then from \eqref{2.11} and \(Q_0 \neq 0\), we get \(f \equiv 0\), a contradiction. $v=0$, thus \eqref{2.16} gives 
\begin{align}
\label{2.18}
0 = 2abP Q^2_2 + [3bP Q'_0 - (bP' + 2bc'P)Q_0]Q_2 -3bP Q_0 Q'_2 - 3bP Q_0 Q_1. 
\end{align}

From \eqref{2.10} and \eqref{2.12},  we obtain
\begin{equation}
\begin{aligned}
\label{2.19}
Q_2 &= -12bP h + \alpha_1, \\
Q_1 &= (4bc'P + 3b'P + 5bP')h - 3bP h' + \beta_2.
\end{aligned}  
\end{equation}
where  
$\alpha_1 = 3bP(bP'' + 4b'c'P + 4bc''P + b''P)  - (bP' + 2bc'P)(4bc'P + 3b'P + 5bP')$ and
$\beta_2 = - (d'P - dP')(4bc'P + 3b'P + 5bP') + 3bP(-4abP + d''P - dP'').$
\eqref{2.18} and \eqref{2.19} implies
\begin{equation}
\label{2.20}
288ab^3P^3h^2 + U_1h' + U_2h + U_3 = 0,  
\end{equation}
where
$U_1 = 45b^2P^2Q_0, $ 
$U_2 = -48ab^2P^2\alpha_1 - 36b^2P^2Q'_0    + 3bP Q_0(11bP' + 9b'P + 4bc'P),$  
$U_3= 2abP\alpha_1^2 + \left[3bP Q'_0 - (bP' + 2bc'P)Q_0\right]\alpha_1 
 - 3bP Q_0(\alpha'_1 + \beta_2). $
By clunie lemma, $\eqref{2.20}$ and $N(r,h)=O(\log r)$,  we get $h$ is  a rational function. 

Thus \(Q_0, Q_1, Q_2\) are all rational functions. 
 \eqref{2.11} implies
\begin{equation}
\label{2.21}
f'' = p f' + q f,
\end{equation}
where 
$p = -\frac{Q_1}{Q_2}, q = -\frac{Q_0}{Q_2} .$

By   \eqref{2.21},  \eqref{1.1} can be written as
\begin{equation}
\label{2.22}
be^{2c}= f'f'' - af^2 = p(f')^2 + qff' - af^2.  
\end{equation}
Differentiating \eqref{2.22}, and then using \(f'' = pf' + qf\) and $\eqref{1.1}$, we obtain
\begin{equation}
\label{2.23}
A_0(f')^2 + A_1ff' + A_2f^2 = 0,  
\end{equation}
where
\begin{equation*}
\begin{aligned}
A_0 &= p' + 2p^2 + q - \sigma p, \\
A_1 &= q' + 3pq - 2a - \sigma q, \\
A_2 &= q^2 - a' + \sigma a,\\
\sigma &= \frac{b'}{b} + 2c'.
\end{aligned}  
\end{equation*}

If \(A_0, A_1, A_2\) are not all identically zero, then $N(r,\frac{1}{f})=O(\log r)$. Hence
\[
f = R_0 e^W,
\]
where    $R_0, W  $ are polynomial.
Repeating the argument as \eqref{2.17}, we obtain  conclusion (I).

Next, we consider \(A_0, A_1, A_2\) are   all identically zero, that is 

\begin{equation}
\label{2.24}
\begin{aligned}
0 &= p' + 2p^2 + q - \sigma p, \\
0 &= q' + 3pq - 2a - \sigma q, \\
0 &= q^2 - a' + \sigma a.\\
\end{aligned}  
\end{equation}
Note that \(p \neq 0\), otherwise the first equation gives \(q \equiv 0\), which contradicts \(q = -Q_0/Q_2 \neq 0\).

Define
\begin{equation}
\label{2.25}
S = -\frac{q}{p}, \qquad X = \frac{q^2 + 4ap}{p^2}.
\end{equation}
From the definition it immediately follows that
\begin{equation}
\label{2.26}
q = -pS, \qquad a = \frac{p}{4}(X - S^2).  
\end{equation}
From the first equation of \eqref{2.24} and \(q = -pS\), we obtain
\begin{equation}
\label{2.27}
\sigma = \frac{p'}{p} + 2p - S.  
\end{equation}
Substituting \eqref{2.26} into the second equation of \eqref{2.24}:
\[
-p'S - pS' - 3p^2 S - \frac{p}{2}(X - S^2) = -\sigma p S.
\]
Dividing by \(p\), and then using \eqref{2.27},   we get
\begin{equation}
\label{2.28}
S' = -pS - \frac{S^2 + X}{2}. 
\end{equation}
On the other hand, substituting \eqref{2.26} into the third equation of \eqref{2.23}, and differentiating \(a = p(X - S^2)/4\), we obtain

\[
4pS^2 - \frac{p'}{p}(X - S^2) - X' + 2SS' + \sigma(X - S^2) = 0.
\]
Then simplifying by using \eqref{2.27} and \eqref{2.28}, we obtain
\begin{equation}
\label{2.29}
X' = 2(p - S)X.  
\end{equation}
Finally, from \eqref{2.29}, \eqref{2.28}, and \(8a = 2p(X - S^2)\), direct calculation yields
\begin{equation}
\label{2.30}
8a = 2SS' + X' + S^3 + 3SX.  
\end{equation}
  In the following, we divide into two cases according to \(X \not \equiv 0\) and \(X \equiv 0\).

\textbf{Case 1 \(X \not \equiv 0\).}   From \eqref{2.29} we get \[
p = S + \frac{X'}{2X}.
\]
Substituting  it into  \eqref{2.28},  we obtain
\begin{equation}
\label{2.31}
S X' + (2S' + 3S^2)X + X^2 = 0.  
\end{equation}
Then from \eqref{2.29} and \eqref{2.27},
\[
\sigma - S = \frac{p'}{p} + 2p - 2S = \frac{(pX)'}{pX}.
\]
Substituting further
\[
\sigma = \frac{b'}{b} + 2c',
\]
we have
\begin{align*}
S &= 2c' + \frac{b'}{b} - \frac{(pX)'}{pX} \\
&= 2c' + \frac{b'}{b} - \frac{p'}{p} - \frac{X'}{X}.
\end{align*}
By 
\[
\frac{b'}{b} - \frac{p'}{p} - \frac{X'}{X} = \frac{\left(\frac{b}{pX}\right)'}{\frac{b}{pX}}.
\]
  we   obtain
\begin{equation}
\label{2.32}
S = 2c' + \frac{G'}{G},
\end{equation}
where  $G = \frac{b}{pX}$.

Identities  \eqref{2.30}, \eqref{2.31}, \eqref{2.32} exactly satisfy all the conditions of Lemma \ref{l1}. Hence there exists  nonzero constant $s$, such that
\begin{equation*}
S = s, \quad X = -3s^2, \quad c' = \frac{s}{2}
\end{equation*}
Since \(X\) is a nonzero constant, \eqref{2.29} gives \(p = S = s\). Then from \eqref{2.25} and \eqref{2.26},
\begin{equation*}
q = -s^2, \quad a = -s^3, \quad \sigma = s. 
\end{equation*}
Also since \(G = b/(pX)\) and \(G, p, X\) are all nonzero constants, \(b\) is also a nonzero constant $\mu$. Let \(\lambda = s/2\), then  
\begin{equation*}
c(z) = \lambda z \quad b \equiv \mu, \quad a = -8\lambda^3, \quad p = 2\lambda, \quad q = -4\lambda^2.  
\end{equation*}

\eqref{2.21} can rewrite as  \[
f'' = 2\lambda f' - 4\lambda^2 f.
\]
Its characteristic equation and two distinct characteristic roots are
\[
t^2 - 2\lambda t + 4\lambda^2 = 0,\qquad t_{1,2} = (1 \pm i\sqrt{3})\lambda.
\]
Therefore
\[
f = \alpha e^{t_1 z} + \beta e^{t_2 z} \qquad (\alpha,\beta \in \mathbb{C}).
\]
Substituting it into \eqref{2.22},

$\mu e^{2\lambda z} = 2\lambda \alpha^2 (t_1^2 - 2\lambda t_1 + 4\lambda^2) e^{2t_1 z} + 2\lambda \beta^2 (t_2^2 - 2\lambda t_2 + 4\lambda^2) e^{2t_2 z} + 4\lambda \alpha\beta (t_1 t_2 - \lambda(t_1+t_2) + 4\lambda^2) e^{(t_1+t_2)z}$. 
By
$t_i^2 - 2\lambda t_i + 4\lambda^2 = 0$, we have $\mu e^{2\lambda z}=4\lambda \alpha\beta (t_1 t_2 - \lambda(t_1+t_2) + 4\lambda^2) e^{(t_1+t_2)z}$.
Substituting  \(t_1 + t_2 = 2\lambda\) and \(t_1 t_2 = 4\lambda^2\) into this equation, we get
\[
\mu e^{2\lambda z} = 24\alpha\beta\lambda^3e^{2\lambda z}.
\]
Thus, 
$\mu = 24\alpha\beta\lambda^3.$
Since \(\mu \neq 0\), we must have \(\alpha\beta \neq 0\). This is conclusion  (III).

\textbf{Case 2 \(X\equiv 0\).}
 Define 
\begin{equation}
r = \frac{q}{2p}. 
\end{equation}
Since \(S = -q/p = -2r\), from \(X = 0\) and \eqref{2.26} we get \(q = 2pr\), \(a = -pr^2\), and \eqref{2.22} imples
\[
be^{2c}=p(f' + rf)^2.
\]
Because \(a \neq 0\), we have \(r \neq 0\). Substituting \(S = -2r\), \(X = 0\) into \eqref{2.28}, we obtain
\begin{equation}
\label{3.35}
p = r - \frac{r'}{r}, \qquad a = rr' - r^3. 
\end{equation}

Then from \eqref{2.27} we   obtaining
\begin{equation}
\label{3.36}
\sigma = 4r + \frac{H'}{H}.  
\end{equation}
where 
$H = \frac{r^2 - r'}{r^3}.$ Let \(r_0\) be the polynomial part of \(r\). Comparing the polynomial parts of \(\sigma = b'/b + 2c'\) and \eqref{3.36}, we get
$c' = 2r_0.$
 
Thus we obtain 
\begin{equation}
\label{3.37}
\frac{b'}{b} = 4(r - r_0) + \frac{H'}{H}.
\end{equation}
If \(z_0\) is a zero of \(r\) of multiplicity \(m\), then
\[
\operatorname{ord}_{z_0}(r^2 - r') = m - 1, \qquad \operatorname{ord}_{z_0} H = -(2m + 1).
\]
At \(z_0\), \(r-r_0\) is analytic, while the residue of \(H'/H\) is \(-(2m+1)<0\). Hence the residue of the right-hand side of \eqref{3.37} at \(z_0\) is a negative integer. But \(b\) is a nonzero polynomial, and the residue of \(b'/b\) at any finite point can only be the multiplicity of a zero of \(b\), and therefore must be a nonnegative integer, a contradiction. Thus \(r\) has no finite zeros.

Write \(r=P_1/P_2\) as a reduced rational expression. Having no finite zeros means that \(P_1\) is a nonzero constant; while  \(r_0\not\equiv0\), i.e. \(\deg P_1\ge\deg P_2\). Therefore \(P_2\) is also constant, so \(r\) is a nonzero constant. Let
\begin{equation*}
\lambda = 2r.  
\end{equation*}

From \eqref{3.35}, \eqref{3.36}, and \eqref{3.37},
\begin{equation}
c(z) = \lambda z, \quad b \equiv \mu, \quad p = \frac{\lambda}{2}, \quad q = \frac{\lambda^2}{2}, \quad a = -\frac{\lambda^3}{8}. 
\end{equation}
At this point, \eqref{2.21} is 
\[
f'' = \frac{\lambda}{2} f' + \frac{\lambda^2}{2} f.
\]
Its characteristic equation is
\[
t^2 - \frac{\lambda}{2} t - \frac{\lambda^2}{2} = 0,
\]
and the characteristic roots are \(\lambda\) and \(-\lambda/2\), so
\[
f =    \alpha e^{\lambda z} + \beta e^{-\lambda z/2}.
\]
\eqref{2.22} gives 
\[
\mu e^{2\lambda z}= p(f' + rf)^2 = \frac{\lambda}{2} \left( f' + \frac{\lambda}{2} f \right)^2.
\]
We get
$\mu = \frac{9}{8}\alpha^2\lambda^3.$
Therefore \(\alpha \neq 0\). If \(\beta = 0\), this is conclusion (I); if \(\beta \neq 0\), it is conclusion (II).  

\qed

	\section*{Declarations}
	\begin{itemize}
		\item \noindent{\bf Funding}
		This research work was supported by Graduate Research Fund Project of Guizhou Province ((2025YJSKYJJ107)) and the National Natural Science Foundation of China (Grant No. 12261023, 11861023).
		
		\item \noindent{\bf Conflicts of Interest}
		The authors declare that there are no conflicts of interest regarding the publication of this paper.

        \item\noindent{\bf Author Contributions}
All authors contributed to the study conception and design. All authors read and approved the final manuscript.
	\end{itemize}

\end{document}